\PassOptionsToPackage{dvipsnames}{xcolor}
\documentclass{amsart}
\usepackage[a4paper, top=2cm, bottom=3cm, left=2.5cm, right=2.5cm]{geometry}
\usepackage[utf8]{inputenc}
\usepackage[english]{babel}
\usepackage{bm}
\usepackage{mathrsfs}
\usepackage{amsmath}
\usepackage{bbm}
\usepackage{amssymb}
\usepackage{enumerate}
\usepackage{amsthm}
\usepackage{graphicx}
\usepackage{enumitem}
\usepackage{mathtools}
\usepackage{xcolor}
\usepackage{float}
\usepackage[colorlinks = true, urlcolor = blue, citecolor = PineGreen, linkcolor = blue]{hyperref}
\usepackage[most]{tcolorbox}
\usepackage{tikz-cd}
\usepackage{pgfplots}
\usepackage{tikz,tikz-3dplot}
\usepackage{stackengine}   
\usepackage{scalerel}
\usepackage{footnote}
\usepackage{etoolbox}
\usepackage[bottom]{footmisc}
\usepackage{thmtools}
\usepackage{thm-restate}
\usepackage{etoc}
\usepackage{titletoc}
\usepackage{csquotes}
\usepackage{lipsum}
\usepackage[toc,page]{appendix}
\usepackage{subcaption}
\usepackage[style=alphabetic, sorting = nyt, backend=biber, maxnames=5, maxalphanames=3, minalphanames=3 , giveninits=true]{biblatex}

\DeclareNolabel{
 \nolabel{\regexp{[\p{P}\p{S}\p{C}\p{Z}]+}}
}
\usepackage{setspace}
\usepackage{aligned-overset}
\colorlet{luigi}{green!60!gray}
\colorlet{author1}{blue!60!gray}
\colorlet{author2}{red!60!gray}
\usetikzlibrary{arrows,decorations.markings,intersections,patterns,shadings,calc,shapes,positioning }
\usetikzlibrary{shapes.geometric}
\usetikzlibrary{angles, quotes}
\usetikzlibrary{fadings}
\pgfplotsset{compat=newest}
\usepgfplotslibrary{fillbetween}
\numberwithin{equation}{section}
\mathtoolsset{showonlyrefs}
\setlist[description]{leftmargin=4mm}

\setlist[itemize]{leftmargin=5mm}
\setlist[enumerate]{leftmargin=6mm}
\allowdisplaybreaks
\newcounter{count}
\BeforeBeginEnvironment{tcolorbox}{\savenotes}
\AfterEndEnvironment{tcolorbox}{\spewnotes}
\theoremstyle{plain}
\newtheorem{theorem}{Theorem}[section]

\newtheorem{lemma}[theorem]{Lemma}

\newtheorem{conjecture}[theorem]{Conjecture}
\newtheorem*{theorem*}{Theorem}
\newtheorem*{corollary*}{Corollary}
\newtheorem*{proposition*}{Proposition}
\newtheorem*{lemma*}{Lemma}
\newtheorem*{exercise*}{Exercise}
\theoremstyle{definition}
\newtheorem{definition}[theorem]{Definition}

\newtheorem*{definition*}{Definition}
\newtheorem*{example*}{Example}
\newtheorem*{question*}{Question}
\theoremstyle{remark}
\newtheorem{remark}[theorem]{Remark}
\newtheorem*{remark*}{Remark}
\makeatletter
\renewenvironment{proof}[1][\proofname]{\par
  \pushQED{\qed}%
  \normalfont \topsep6\p@\@plus6\p@\relax
  \trivlist
  \item[\hskip\labelsep
        \bfseries
    #1\@addpunct{.}]\ignorespaces
}{%
  \popQED\endtrivlist\@endpefalse
}
\makeatother
\newcommand{\R}{\mathbb{R}}
\newcommand{\N}{\mathbb{N}}

\newcommand*\dif{\mathop{}\!\mathrm{d}}

\DeclareMathOperator{\dist}{dist}
\DeclareMathOperator{\lip}{Lip}
\DeclareMathOperator{\supp}{supp}
\title{A refined Lusin type theorem for gradients}

\author[L.~De Masi]{Luigi De Masi}
\address{\textit{L.~De Masi:} Dipartimento di Matematica, Università di Trento, Via Sommarive 14, 38123 Trento, Italy}
\email{luigi.demasi@unitn.it}

\author[A.~Marchese]{Andrea Marchese}
\address{\textit{A.~Marchese:} Dipartimento di Matematica, Università di Trento, Via Sommarive 14, 38123 Trento, Italy}
\email{andrea.marchese@unitn.it}

\thanks{{\bf Acknowledgments.} }

\subjclass[2010]{}
\begin{document}
	\begin{abstract}
		We prove a refined version of the celebrated Lusin type theorem for gradients by Alberti, stating that any Borel vector field $f$ coincides with the gradient of a $C^1$ function $g$, outside a set $E$ of arbitrarily small Lebesgue measure. We replace the Lebesgue measure with any Radon measure $\mu$, and we obtain that the estimate on the $L^p$ norm of $Dg$ does not depend on $\mu(E)$, if the value of $f$ is $\mu$-a.e. orthogonal to the decomposability bundle of $\mu$. We observe that our result implies the 1-dimensional version of the flat chain conjecture by Ambrosio and Kirchheim on the equivalence between metric currents and flat chains with finite mass in $\mathbb{R}^n$ and we state a suitable generalization for $k$-forms, which would imply the validity of the conjecture in full generality. 
	\end{abstract}

	\maketitle

\section{Introduction}

In this note, we improve the \emph{Lusin type theorem for gradients} by Alberti, see \cite{alberti}, and its generalization for Radon measures, see Theorem 2.1 of \cite{marchese_schioppalipschitz}, by proving the following result. We recall that the decomposability bundle of a Radon measure $\mu$ on $\R^n$, is a Borel map $V(\mu,\cdot):\R^n\to\bigcup_{k\leq n}Gr(k,n)$, taking as values vector subspaces of $\R^n$, see Section \ref{s:notation} for the definition. Given a Borel map $f:\R^n\to\R^n$, we denote by $f_V$ the projection of $f$ on $V(\mu,\cdot)$, that is, the map $f_V:x\mapsto P_{V(\mu,x)}f(x)$, where, for every subspace $L$ of $\R^n$, $P_L$ denotes the orthogonal projection onto $L$. Similarly $f_{V^\perp}$ denotes the projection onto the orthogonal complement of $V(\mu,\cdot)$.

\begin{theorem}\label{t:main}
    Let $\mu$ be a Radon measure on $\R^n$, let $\Omega \subset \R^n$ be an open set such that $\mu(\Omega)< \infty$ and let $f \colon \Omega \to \R^n$ be a Borel map. There exists a constant $C(n)$ such that for every $\varepsilon>0$ there exists a compact set $K \subset \Omega$ and a function $g \colon \Omega \to \R$ of class $C^1$ such that,
    \begin{gather}
        \mu(\Omega \setminus K) < \varepsilon,
        \\[4pt]
        Dg(x)=f(x) \qquad \forall x \in K,
        \\[4pt]
        \|Dg\|_{L^p(\mu)} \leq C\varepsilon^{\frac{1}{p}-1}\|f_V\|_{L^p(\mu)}+ (1+\varepsilon)\|f_{V^\perp}\|_{L^p(\mu)}, \quad\mbox{for every $p\in[1,\infty]$}.
    \end{gather}
\end{theorem}
\begin{remark}
\begin{itemize}
    \item[(i)] The main result of \cite{alberti}, corresponds to the choice in which $\mu$ is the Lebesgue measure. In this case $V(\mu,\cdot)\equiv\R^n$, so that the second addendum in the estimate for $\|Dg\|_{L^p(\mu)}$ vanishes.
    \item[(ii)] Our result improves also Theorem 2.1 of \cite{marchese_schioppalipschitz}, providing an estimate for $\|Dg\|_{L^p(\mu)}$ which remains bounded when $\varepsilon\to 0$, if $f(x)$ is orthogonal to $V(\mu,x)$, for $\mu$-a.e. $x\in\Omega$. 
    \item[(iii)]In the latter case, since the Lipschitz constant of $g$ does not depend on $\varepsilon$, one could be tempted to conclude that a stronger conclusion holds. Namely, that there exists a Lipschitz function $g$ such that $D_{f(x)}g(x)=1$, for $\mu$-a.e. $x\in\Omega$. However, this conclusion is false in general. Indeed the natural choice for $g$ could be a subsequential limit as $\varepsilon\to 0$ of the functions $g_\varepsilon$ (depending on $\varepsilon)$ yielded by Theorem \ref{t:main}, but being the limit $g$ only Lipschitz, its directional derivative in direction $f(x)$ could fail to exist, because $f(x)\not\in V(\mu,x)$ for $\mu$-a.e $x$, see Theorem 1.1 (ii) of\cite{AlbMar}. Even in the case in which $D_{f(x)}g(x)$ exists at $\mu$-a.e. $x$, the conclusion above would still not be legitimate, because this directional derivative is not closable from the space of Lipschitz functions to $(L^\infty(\mu),w^*)$, see Theorem 1.1 of \cite{alberti-bate-marchese}.  
\end{itemize}    
\end{remark}
In Section \ref{s:metric}, we state a Conjecture which is a suitable generalization of Theorem \ref{t:main} for $k$-forms and we show that its validity implies the validity of the Ambrosio-Kirchheim flat chain conjecture, see \cite{AK}, modifying a strategy outlined in \cite{MarMer}. In particular we observe that Theorem \ref{t:main} implies the validity of the conjecture in dimension $k=1$.

\subsection*{Acknowledgements}
The authors are supported by INdAM-GNAMPA, through the project "Minimizers for the area functional: existence, regularity and geometrical properties" and by the project PRIN 2022PJ9EFL "Geometric Measure Theory: Structure of Singular Measures, Regularity Theory and Applications in the Calculus of Variations", funded by the European Union - Next Generation EU, Mission 4
Component 2 - CUP:E53D23005860006.

\section{Notation and preliminaries}\label{s:notation}
For any (possibly vector valued) Radon measure $\mu$ on $\R^n$ we denote by $|\mu|$ its total variation measure and by $\mathbb{M}(\mu):=|\mu|(\R^n)$. For any Borel set $E\subset\R^n$, we denote by $\mu\llcorner E$ the restriction of $\mu$ to $E$, that is, the measure $A\mapsto \mu(E\cap A)$, for every Borel set $A$. 

If $e \in \mathbb{S}^{n-1}$ and $\alpha \in \left[0,\frac{\pi}{2}\right]$, we denote by $C(e,\alpha)$ the convex cone with axis $e$ and opening angle $\alpha$, namely
\begin{equation}
    C(e,\alpha)
    \coloneqq
    \{x \in \R^n \colon x \cdot e \geq \cos \alpha |x|\}.
\end{equation}

We recall the following definition, see \S 2.6, \S 6.1 and Theorem 6.4 of \cite{AlbMar}. 

\begin{definition}[Decomposability bundle]
Given a positive Radon measure $\mu$ on $\R^n$ its \emph{decomposability bundle} is any Borel map $V(\mu,\cdot)$ taking values in the set $\bigcup_{k\leq n}Gr(k,n)$, that is, in the union of the Grasmannians of vector subspaces of $\R^n$ of any dimension, with the following property. For $\mu$-a.e. $x\in\R^n$, a vector $v\in\R^n$ belongs to $V(\mu,x)$ if and only if there exists a vector-valued measure $T$ with ${\mathrm {div}} T=0$ such that
$$\lim_{r\to 0}\frac{\mathbb{M}((T-v\mu)\llcorner B(x,r))}{\mu(B(x,r))}=0.$$
\end{definition}

\section{Proof of the main result}

Thanks to Theorem 2.1 of \cite{marchese_schioppalipschitz}, in order to prove Theorem \ref{t:main}, it is sufficient to prove the following result.

\begin{theorem}\label{t:main2}
    Let $\mu$ be a Radon measure on $\R^n$, let $\Omega \subset \R^n$ be an open set such that $\mu(\Omega)< \infty$ and let $f \colon \Omega \to \R^n$ be a Borel function satisfying $f(x) \in V(\mu,x)^\perp$ for $\mu$-a.e.\ $x \in \Omega$. For every $\varepsilon>0$ there exists a compact set $K \subset \Omega$ and a function $g \colon \Omega \to \R$ of class $C^1$ such that
    \begin{gather}
        \mu(\Omega \setminus K) < \varepsilon,
        \\[4pt]
        Dg(x)=f(x) \qquad \forall x \in K,
        \\[4pt]
        \|Dg\|_{L^p(\mu)} \leq (1+\varepsilon)\|f\|_{L^p(\mu)}, \quad\mbox{for every $p\in[1,\infty]$}.
    \end{gather}
\end{theorem}

The proof is based on the following lemmas.

\begin{lemma}[quantitative version of Lusin's theorem]\label{lmm:quantitative_Lusin}
    Let $\mu$ be a Radon measure on $\R^n$, let $\Omega \subset \R^n$ be an open set such that $\mu(\Omega)< \infty$ and let $f \colon \Omega \to \R^n$ be a Borel function. Then, for every $\varepsilon>0$, there exists $h \in C_c(\Omega, \R^n)$ such that
    \begin{gather}
        \mu \left( \{x \in \Omega \colon h(x) \neq f(x)\} \right) < \varepsilon,
        \\[5pt]
        \| h\|_{L^p(\Omega)} \leq (1+\varepsilon) \|f\|_{L^p(\Omega)}
        \qquad
        \forall p \in [1,+\infty].
    \end{gather}
\end{lemma}

\begin{remark}
    Notice that $h$ depends only on $f$ and $\varepsilon$ and does not depend on $p$.
\end{remark}

\begin{proof}[Proof of lemma \ref{lmm:quantitative_Lusin}]
    Of course we can assume $f \not \equiv 0$, otherwise there is nothing to prove, and $\varepsilon<1$. Moreover, we can assume that $f \in L^p(\Omega)$ for some $p \in [1,+\infty]$, otherwise the thesis is just the classical Lusin's theorem, up to performing a cut-off of the resulting continuous function in a small neighborhood of $\partial \Omega$ (see part 3 of this proof).
    
    For every $t > 0$ we call
    \begin{equation}
        E(f,t) \coloneqq \{x \in \Omega \colon |f(x)|\geq t \}.
    \end{equation}
    \begin{description}
        \item[1 - Truncation.]
        We first observe that, since $f \in L^p(\Omega)$ for some $p \in [1,+\infty]$, there exists $t>0$ such that either $0<\mu\big(E(f,t)\big)< \frac{\varepsilon}{3}$ or
        \begin{equation}\label{eq:misura>eps}
            \mu\big(E(f,t)\big) \geq \frac{\varepsilon}{3}
            \qquad
            \text{and}
            \qquad
            \mu\big(E(f,s)\big) = 0 \quad \forall s>t.
        \end{equation}
        In either case, we call $E \coloneqq E(f,t)$ and we construct $f_1$ by projecting $f$ on the sphere of radius $t$, namely we define
        \begin{equation}
            f_1(x)
            =
            \begin{cases}
                f(x)                        & \text{if } |f(x)| < t
                \\[8pt]
                t \dfrac{f(x)}{|f(x)|}      & \text{if } |f(x)| \geq t.
            \end{cases}
        \end{equation}
        Clearly $\|f_1\|_{L^\infty(\Omega)} = t$ and, since $|f_1(x)| \leq |f(x)|$ for every $x \in \Omega$, we have
        \begin{gather}
            \|f_1\|_{L^p(\Omega)} \leq  \|f\|_{L^p(\Omega)}
            \qquad
            \forall p \in [1,+\infty].
        \end{gather}
        Moreover $f_1=f$ $\mu$-a.e. in case \eqref{eq:misura>eps} holds, whereas
        \begin{equation}\label{eq:difference_f_1_f}
            \mu \big( \{x \in \Omega \colon f_1(x) \neq f(x)\}\big)
            \leq \mu(E)
            < \frac{\varepsilon}{3},
        \end{equation}
        in the other case.
        \item[2 - Continuous approximation.]
        Chosing $0 < \eta < \varepsilon \min\left\{\mu(E),\frac{\varepsilon}{3}\right\}$, by classical Lusin's theorem there exists a continuous function $f_2 \colon \Omega \to \R$ such that $F \coloneqq \{x \in \Omega \colon f_1(x) \neq f_2(x) \}$ satisfies
        \begin{equation}
            \mu(F) < \eta.
        \end{equation}
        Up to projecting on $B_t(0)$, we can assume $\|f_2\|_{C(\Omega)} \leq t$, thus
        \begin{equation}\label{eq:norm_infty_f_2}
            \|f_2\|_{L^\infty(\Omega)} \leq \|f_2\|_{C(\Omega)}  \leq t = \|f_1\|_{L^\infty(\Omega)} \leq \|f\|_{L^\infty(\Omega)}.
        \end{equation}
        We now estimate the difference $\|f_2\|_{L^{p}(\Omega)} - \|f\|_{L^{p}(\Omega)}$ uniformly for $p \in [1,+\infty)$. We have
        \begin{equation}
            \int_{\Omega} |f_2(x)|^p \dif \mu(x)
            \leq
            \int_{\Omega \setminus (F \cup E)} |f(x)|^p \dif \mu(x)
            +
            t^p \mu(F \cup E),
        \end{equation}
        while
        \begin{equation}
            \int_{\Omega} |f(x)|^p \dif \mu(x)
            \geq
            \int_{\Omega \setminus (F \cup E)} |f(x)|^p \dif \mu(x)
            +
            t^p \mu(E).
        \end{equation}
        By these estimates, using $\mu(F) < \varepsilon \mu(E)$ and the relation (due to the concavity of the function $s \mapsto s^{\frac{1}{p}}$ and $p \geq 1$)
        \begin{equation}
            \big(a+(1+\varepsilon)b \big)^{\frac{1}{p}} - (a+b)^{\frac{1}{p}}
            \leq
            b^{\frac{1}{p}}\big[ (1+\varepsilon)^{\frac{1}{p}} - 1 \big]
            \leq
            \varepsilon b^{\frac{1}{p}}
            \qquad
            \forall a,b \geq 0, \,\,\, \forall p \in [1,+\infty),
        \end{equation}
        we obtain
        \begin{equation}\label{eq:norm_p_f_2}
            \|f_2\|_{L^{p}(\Omega)} - \|f\|_{L^{p}(\Omega)}
           \leq
           \varepsilon t \mu(E)^{\frac{1}{p}}
           \leq
           \varepsilon \|f\|_{L^p(\Omega)}
           \qquad
           \forall p \in [1,+\infty].
        \end{equation}
        Gathering \eqref{eq:norm_infty_f_2} and \eqref{eq:norm_p_f_2} we get
        \begin{equation}\label{eq:norm_f_2}
            \|f_2\|_{L^{p}(\Omega)} \leq (1+\varepsilon) \|f\|_{L^{p}(\Omega)}
            \qquad
            \forall p \in [1,+\infty].
        \end{equation}
        We moreover record that the choice of $\eta$ and \eqref{eq:difference_f_1_f} implies
        \begin{equation}\label{eq:measure_f_2}
           \mu \big( \{x \in \Omega \colon f_2(x) \neq f(x) \} \big) < \frac{2}{3} \varepsilon.
        \end{equation}

        \item[3 - Cut-off.]
        Since $\Omega$ is open and $\mu$ is a Radon measure, there exists $\rho>0$ such that
        \begin{equation}\label{eq:measure_cut-off}
            \mu\Big(\big\{x \in \Omega \colon \dist(x, \partial \Omega) \leq \rho \big\} \Big)
            <
            \frac{\varepsilon}{3}.
        \end{equation}
        We choose any cut-off function $\chi \in C_c(\Omega)$ such that $\chi(x) = 1$ if $\dist(x,\partial \Omega) \geq \rho$ and we define $h \coloneqq \chi f_2$.  By $|h| \leq |f_2|$, \eqref{eq:norm_f_2}, \eqref{eq:measure_f_2} and \eqref{eq:measure_cut-off}, $h$ is the desired function.
    \end{description}  
\end{proof}

\begin{remark}
    If $f$ is such that
    \begin{equation}
                \forall \eta>0 \,\,\, \exists t>0 \colon \quad \eta > \mu\big(E(f,t)\big) >0,
    \end{equation}
    then by a similar procedure one can in fact construct a continuous function $h$ with compact support which differs from $f$ on a small set and
    \begin{equation}
        \| h\|_{L^p(\Omega)} \leq \|f\|_{L^p(\Omega)}
        \qquad
        \forall p \in [1,+\infty].
    \end{equation}
\end{remark}

We recall that $C(e,\alpha)$ is the convex cone with axis $e$ and angle $\alpha$, see section \ref{s:notation}.

\begin{lemma}\label{lmm:basic_step}
    Let $\mu$ be a Radon measure on an open set $\Omega\subset\R^n$, let $B\subset\subset\Omega$ be an open ball and assume that there exists $e\in\mathbb{S}^{n-1}$ and $\alpha>0$ such that $V(\mu,x)\cap C(e,\alpha)=0$ for $\mu$-a.e. $x\in B$. Let $\delta>0$ and let $f \colon \Omega \to \R^n$ be a Borel function such that
    \begin{equation}\label{e:parallel}
        \big|f(x)-\|f\|_{L^\infty(B)} e\big|\leq \delta \qquad \text{for every $x \in B$.}
    \end{equation}  
    Then, for every $\eta>0, \gamma>0$ there exist an open set $U\subset B$ and a function $w \colon \Omega\to \R$ of class $C^1$ such that
    \begin{gather}
        \sup_B|w|\leq \eta,
        \\[4pt]
        \mu(B\setminus U)\leq\gamma,
        \\[4pt]
        |Dw(x)-f(x)|\leq 4\delta+\frac{\|f\|_{\infty}}{\tan \alpha} \qquad \mbox{ for every $x \in U$},
        \\[4pt]
        |Dw(x)|\leq  \|f\|_{L^\infty(B)} \left(1 + \frac{1}{\tan\alpha}\right) \qquad \mbox{ for every $x \in B$}.
    \end{gather}
\end{lemma}

\begin{proof}
By Lemma 7.5 of \cite{AlbMar}, there exists a $C(e,\alpha)$-null Borel set $F\subset B$ such that $\mu(B\setminus F)=0$. Hence we can find a compact subset $K\subset F$ such that $\mu(B\setminus K)<\gamma$. Let $w_0$ be the smooth function constructed in Lemma 4.12 of \cite{AlbMar}, with $\varepsilon=\frac{\eta}{\|f\|_\infty+\delta}$. 
Combining properties (ii) and (iii) of Lemma 4.12 of \cite{AlbMar}, we deduce that
$$\left|Dw_0(x)- e \right|<\frac{1}{\tan \alpha},\qquad\mbox{for every $x\in K$}$$
and
$$|Dw_0(x)|<1+\frac{1}{\tan \alpha},\qquad\mbox{for every $x\in \R^n$}.$$
By \eqref{e:parallel} the function $w:=\|f\|_\infty w_0$ satisfies:
    \begin{gather}
        \sup_B|w|\leq \eta,
        \\[4pt]
        |Dw(x)-f(x)|\leq \big|Dw(x)-\|f\|_{L^\infty(B)} e\big|+\big|\|f\|_{L^\infty(B)} e-f(x)\big| \leq \frac{\|f\|_{L^\infty(B)}}{\tan\alpha} + \delta  \qquad \mbox{ for every $x \in K$},
        \\[4pt]
        |Dw(x)|\leq \|f\|_{L^\infty(B)} \left(1 + \frac{1}{\tan\alpha}\right) \qquad \mbox{ for every $x \in B$}.
    \end{gather}
    Since $w$ is of class $C^1$, we can take a sufficiently small  open tubular neighbourhood $U$ of $K$ in $B$ so that, for every $x \in U$, denoting $y$ a point in $K$ which minimizes the distance to $x$, we have
    \begin{equation}
    \begin{split}
        |Dw(x)-f(x)|&\leq |Dw(x)-Dw(y)|+|Dw(y)-f(y)|+|f(y)-f(x)|\\
        &\leq \delta +\sup_{y\in K}|Dw(y)-f(y)|+\operatorname{osc}_{B}(f)
        \\
        &\leq 4\delta+\frac{\|f\|_{L^\infty(B)}}{\tan \alpha}.
    \end{split}
    \end{equation}
\end{proof}

\begin{lemma}[approximation of continuous functions with compact support]\label{lmm:approx_continuous}
        Let $\mu$ be a Radon measure on $\R^n$, let $\Omega \subset \R^n$ be an open set such that $\mu(\Omega)< \infty$ and let $h \colon \Omega \to \R^n$ be a continuous function with compact support in $\Omega$ satisfying $h(x) \in V(\mu,x)^\perp$ for $\mu$-a.e.\ $x \in \Omega$. Then for every $\varepsilon>0$ there exists a compact set $K \subset \Omega$ and a function $g \colon \Omega \to \R$ of class $C^1$ such that
    \begin{gather}
        \mu(\Omega \setminus K) < \varepsilon,
        \\[4pt]
        \supp g \subseteq \supp h,
        \\[4pt]
        Dg(x)=h(x) \qquad \forall x \in K,
        \\[4pt]
        \|Dg\|_{C^0(\Omega)} \leq (1+\varepsilon) \|h\|_{C^0(\Omega)}.
    \end{gather}
\end{lemma}

\begin{proof}
Of course we can assume $\varepsilon<1$ and $h \not \equiv 0$.
By the classical Lusin's theorem we find a compact set $C$ such that the map $x\mapsto V(\mu,x)$ is continuous on $C$, that  $h(x) \in V(\mu,x)^\perp$ for every $x \in C$ and
\begin{gather}
       \mu(\Omega \setminus C) < \frac{\varepsilon}{4}. \label{eq:misura_C}
    \end{gather}
The idea is to construct a series of $C^1$ functions $g_n$ such that $h- \sum_{k=0}^n Dg_n$ converges uniformly to $0$ on a compact set $K \subset \Omega$ with large measure and define $g \coloneqq \sum_{k=0}^{+\infty} g_n$. More precisely we are going to inductively define a decreasing sequence of open sets $V_n$ and a sequence of functions $g_n \in C^1_c(V_n)$  with the following properties:
  \begin{gather}
      \mu(\Omega \setminus V_n) \leq (1-2^{-n-1})\frac{\varepsilon}{2}
      \qquad
      \forall n \geq 0, \label{eq:measure_V_n}
      \\[5pt]
      \left\|h- \sum_{k=0}^n Dg_n \right\|_{C^0(V_n)} \leq \left( \frac{\varepsilon}{2} \right)^{n} \|h\|_{C^0(\Omega)}
      \qquad
      \forall n \geq 0,   \label{eq:norm_difference_partial_sum}
      \\[4pt]
      \|Dg_n\|_{C^0(\Omega)} \leq \left(1+\frac{\varepsilon}{2} \right) \left( \frac{\varepsilon}{2} \right)^{n} \|h\|_{C^0(\Omega)}
      \qquad
      \forall n \geq 0,    \label{eq:norm_general_term_series}
  \end{gather}
and finally we will fix a suitable compact set $K \subseteq \bigcap_{n \in \N} V_n$ on which $h= \sum_{k=0}^{+\infty} g_n$.
As base of the induction, we set
\begin{equation}
    g_0=0,
    \qquad
    V_0\coloneqq \Omega.
\end{equation}
\begin{description}[style=nextline, font=\normalfont\textit]
    \item[Step 1: First approximation.]
    Let us choose $\delta>0,0<\alpha<\frac{\pi}{2}$ such that 
    \begin{equation}\label{e:scelta}
    4\delta+\frac{\|h\|_{C^0(\Omega)}}{\tan\alpha}\leq \frac{\varepsilon}{2}\|h\|_{C^0(\Omega)}.
    \end{equation}
   
    Let $x \in C$; we distinguish two cases.
    \begin{itemize}
        \item[-] If $h(x)=0$, by continuity there exists $r_x>0$ such that
        \begin{equation}
            |h(y)| \leq \frac{\delta}{2}
            \qquad
            \forall y \in B_{r_x}(x).
        \end{equation}
        \item[-] If $h(x) \neq 0$, we set
        \begin{equation}
            e = \frac{h(x)}{|h(x)|}.
        \end{equation}
        By continuity of $V$ and $h$, and using that $h(x)\in V(\mu,x)^\perp$ for every $x\in C$, we can find $r_x>0$  such that $B_{r_x}(x) \subset \subset \{h \neq 0\}$ and
        \begin{gather}
        \label{eq:continuity_bundle}
            V(\mu,y) \cap C(e,\alpha) = \emptyset,
            \qquad
            \forall y \in C \cap B_{r_x}(x),
            \\
            \label{eq:cone_continuity_f}
            \langle h(y); e \rangle > (1-\delta)|h(y)|
            \big|h(y)-\|f\|_{C^0(B)} e\big|\leq \delta
            \qquad
            \forall y \in B_{r_x}(x).
        \end{gather}
    \end{itemize}
     By classical covering arguments, see for instance \cite[Theorem 1.29]{evans_gariepy}, there exists a finite set of points $x_1,\dots,x_N \in C$ and radii $r_i \in (0,r_{x_i})$ for $i=1,\dots,N$ such that the closed balls $\overline B_{r_i}(x_i) \subset \Omega$ are mutually disjoint and it holds
    \begin{equation}\label{e:tantamassa}
        \mu \left( C \setminus \bigcup_{i=1}^N B_{r_i}(x_i) \right)< 2^{-4}{\varepsilon}.
    \end{equation}
    Let us call $B_i \coloneqq B_{r_i}(x_i)$. Since the closed balls $\overline B_i$ are disjoint and contained in $\Omega$ we can fix $d>0$ satisfying the following conditions:
    \begin{equation}\label{eq:distanza_palle}
    \begin{gathered}
        d <  \frac{1}{8}\min \left\{ \dist\big(B_i,B_j\big) \colon i,j=1,\dots,N, i\neq j \right\},
        \\
        d < \frac{1}{8}\min \{\dist(B_i,\partial \Omega) \colon i=1,\dots,N\},
        \\
        h(x_i) \neq 0 \implies \overline B_{r_i+4d}(x_i) \subset \{h \neq 0\}.
        \end{gathered}
    \end{equation}
    so that the closed balls $\overline B_{r_i+4d}(x_i)$ are disjoint and contained in $\Omega$ as well.
    We now fix
    \begin{equation}
        \eta \leq \delta d.
    \end{equation}
    Now, for every $i=1,\dots,N$, we choose a function $w_i \colon \Omega \to \R$ as follows.
    \begin{itemize}
        \item[-] If $h(x_i)=0$, we set $U_i:=B_i$ and we fix $w_i\equiv 0$. We clearly have
        \begin{gather}
            |Dw_i(y)- h(y)|< \delta,
            \qquad
            \forall y \in C \cap B_i,
            \\
            |Dw_i(y)| \leq \|h\|_{C^0(\Omega)}
            \qquad
            \forall y \in B_i.
        \end{gather}
        \item[-] If $h(x_i) \neq 0$, then by \eqref{eq:continuity_bundle} and \eqref{eq:cone_continuity_f} we can apply Lemma \ref{lmm:basic_step} to the measure $\mu\llcorner (C\cap B_i)$ in order to find, choosing  $\gamma_i < \frac{\varepsilon}{2^4 N}$, open sets $U_i\subset B_i$ and functions $w_i \colon \Omega \to \R$ of class $C^1$ satisfying
        \begin{equation}\label{eq:properties_width}
            \begin{gathered}
            \sup_B|w_i|\leq \eta,
            \\[4pt]
            \mu(C\cap B_i\setminus U_i)<\gamma_i,\\[4pt]
            |Dw_i(x)-h(x)|\leq 4\delta+\frac{\|h\|_{C^0(\Omega)}}{\tan\alpha} \qquad \mbox{ for every $x \in U_i$},
            \\[4pt]
            |Dw_i(x)|\leq\|h\|_{C^0(\Omega)} \left( 1+\frac{1}{\tan\alpha}\right)\qquad \mbox{ for every $x \in B_i$}.
        \end{gathered}
        \end{equation}
    \end{itemize}
  We choose suitable functions $\tilde w_i$ such that
  \begin{gather}
      \tilde w_i\equiv w_i,\qquad \mbox{on $B_i$}\\[4pt]
      \tilde w_i=0, \qquad \mbox{on $\Omega\setminus B_{r_i+d}(x_i)$}\\[4pt]
      |D(\tilde w_i)|\leq\|h\|_{C^0(\Omega)} +4\delta+\frac{1}{\tan\alpha},\qquad\mbox{on $\Omega$},
  \end{gather}
  where the possibility to obtain the last property is guaranteed by the choice of $\eta$.
  More precisely, $\tilde w_i$ is constructed as follows: consider $0<\rho_i<d$ such that in the slightly enlarged ball $B_{r_i+\rho_i}(x_i) \subset B_{r_i+d}(x_i)$ we have
  \begin{equation}
      |D w_i(x)|
      \leq
      \|h\|_{C^0(\Omega)} \left( 1+\frac{1}{\tan\alpha}\right) + \delta,\qquad\mbox{on $B_{r_i+\rho_i}(x_i)$}.
  \end{equation}
  Let us fix a linear function $\varphi_i \colon \R \to \R$ such that $\varphi_i(r_i+\rho_i)=1$ and $\varphi_i(r_i + \rho_i+d)=0$ and then consider the function $z_i$ defined as the radial linear extension of $(w_i)_{| \partial B_{r_i+\rho_i}(x_i)}$ to the annulus ${\overline B}_{r_i+\rho_i+d}(x_i) \setminus B_{r_i+\rho_i}(x_i)$, namely the function defined as
  \begin{equation}
      z_i(x)
      =
      \begin{cases}
          w_i(x)           & \text{if } x \in B_{r_i+\rho_i}(x_i)
          \\[5pt]
          w_i\left(x_i + \frac{x-x_i}{|x-x_i|}r_i+\rho_i \right)\varphi(|x-x_i|)    & \text{if } x \in \overline B_{r_i+\rho_i+d}(x_i) \setminus B_{r_i+\rho_i}(x_i)
          \\[5pt]
          0                 & \text{if } x \in \Omega \setminus \overline B_{r_i+\rho_i+d}(x_i).
      \end{cases}
  \end{equation}
  Then, by the choice of $\eta$, it holds $\lip z_i < \|h\|_{C^0(\Omega)} \left( 1+\frac{1}{\tan\alpha}\right)+2\delta$. By a convolution with a mollifier compactly supported in $B_{\rho_i/2}(0)$, we obtain the desired $\tilde w_i$.

    Let us remark that, by the choice of $d$, the supports of the functions $\tilde w_i$ are mutually disjoint and contained in $\{h \neq 0\} \subset \supp h$. Thus the function $g_1 \colon \Omega \to \R$, defined by $g_1(x):=\sum_{i=1}^N\tilde w_i(x)$ is compactly supported in $\bigcup_{i=1}^N B_{r_i+3d}(x_i)$ and satisfies
  \begin{gather}
      |D g_1(x) - h(x)|\leq 4\delta+\frac{\|h\|_{C^0(\Omega)}}{\tan\alpha} \overset{\eqref{e:scelta}}\leq \frac{\varepsilon}{2}\|h\|_{C^0(\Omega)},
      \qquad
      \mbox{for every}\;x \in \bigcup_{i=1}^N U_{i} 
      \\[4pt]
      |D g_1(x)| \leq \|h\|_{C^0(\Omega)} \left( 1+\frac{1}{\tan\alpha}\right)+2\delta\leq \left(1+\frac{\varepsilon}{2}\right)\|h\|_{C^0(\Omega)}
      \qquad
      \mbox{for every}\; x \in \Omega.
      \\[4pt]
      \supp g \subseteq \bigcup_{i=1}^N B_{r_i+3d}(x_i) \subset \{h \neq 0\} \subset \supp h.
  \end{gather}
 Letting,
  \begin{equation}
    V_1:=\bigcup_{i=1}^N U_{i},
  \end{equation}
we notice that 
\begin{gather}
    |h(x) - Dg_1(x)| \leq \frac{\varepsilon}{2}\|h\|_{C^0(\Omega)} \qquad\mbox{for every $x \in V_1$}
    \\[4pt]
    \mu(C\setminus V_1)<2^{-3}\varepsilon
    \overset{\eqref{eq:misura_C}}{\implies}
    \mu(\Omega \setminus V_1) \leq (1-2^{-2})\frac{\varepsilon}{2},
\end{gather}

  \item[Step 2: Iteration.]
  We now assume that $g_n, V_n$ are defined and satisfy \eqref{eq:measure_V_n}, \eqref{eq:norm_difference_partial_sum} and \eqref{eq:norm_general_term_series}: we want to construct $g_{n+1}, V_{n+1}$.

We apply the same arguments of Step 1 with $h - \sum_{k=0}^n D g_n$ in place of $h$ and $V_n$ in place of $\Omega$;
moreover we choose the new radii $r_i$ in such a way that the new balls $B_i$ are contained in $V_n$ and satisfy
\begin{equation}
     \mu \left( C \cap V_n \setminus \bigcup_{i=1}^N B_{r_i}(x_i) \right)< 2^{-n-3}{\varepsilon}.
\end{equation}
We thus construct the open sets $U_i \subset V_n$, then $V_{n+1}$ their union, and the function $g_{n+1} \in C^1_c(V_n)$ as in the previous step, satisfying the following properties:
  \begin{gather}
  \label{eq:misura_V_n}
        \mu(\Omega \setminus V_{n+1}) \leq (1-2^{-n-2}) \frac{\varepsilon}{2},
        \\[4pt]
      \label{eq:supp_g_n}
      \supp g_{n+1} \subset V_n,
        \\
        \label{eq:stima_grad_diff_n}
      \left|h(x) - \left( \sum_{k=0}^n D g_n(x) \right) - D g_{n+1}(x) \right|
      \leq \frac{\varepsilon}{2} \left\|h(x) -  \sum_{k=0}^n D g_n(x) \right\|_{C^0(\Omega)}
     \overset{\eqref{eq:norm_difference_partial_sum}}{\leq}
     \left( \frac{\varepsilon}{2} \right)^{n+1} \|h\|_{C^0(\Omega)}
      \qquad
      \forall x \in V_{n+1},
      \\[4pt]
      \label{eq:stima_grad_n}
      |D g_{n+1}(x)|
      \leq \left( 1+ \frac{\varepsilon}{2}\right)\left\|h(x) -  \sum_{k=0}^n D g_n(x) \right\|_{C^0(V_n)}
      \overset{\eqref{eq:norm_difference_partial_sum}}{\leq}
      \left( 1+ \frac{\varepsilon}{2}\right) \left( \frac{\varepsilon}{2} \right)^{n} \|h\|_{C^0(\Omega)}
      \qquad
      \forall x \in V_n.
  \end{gather}
  \item[Step 3: Convergence and estimates.]
  By \eqref{eq:measure_V_n} it holds $\mu \left(\Omega \setminus \bigcap_{n \in \N}V_n \right) \leq \frac{\varepsilon}{2}$. Hence there exists a compact set $K \subset \bigcap_{n \in \N}V_n$ such that
  \begin{equation}
      \mu(\Omega \setminus K) < \varepsilon.
  \end{equation}
  We define $g\coloneqq\sum_{i=1}^\infty g_i$ and we observe that the series converges in the strong topology of $C^1(\Omega)$. Indeed each $\supp g_n \subseteq V_{n-1} \subset \Omega$ and, by \eqref{eq:norm_difference_partial_sum}, we get
  \begin{equation}
       \sum_{n=1}^\infty  \|D g_{n}\|_{C^0(\Omega)}
       \leq
       \left( 1+\frac{\varepsilon}{2} \right) \|h\|_{C^0(\Omega)} \sum_{n=0}^\infty \left( \frac{\varepsilon}{2} \right)^{n}
       =
       \left( 1+\frac{\varepsilon}{2} \right) \frac{1}{1- \frac{\varepsilon}{2}} \|h\|_{C^0(\Omega)}
       \overset{(\varepsilon<1)}{<}
       (1+\varepsilon) \|h\|_{C^0(\Omega)}.
  \end{equation}
 This estimate implies the strong convergence of the series $\sum_{n=1}^{+\infty} g_n$ to $g$ in $C^1_c(\Omega)$ and it moreover proves
  \begin{equation}
      \|D g\|_{C^0(\Omega)} \leq (1+\varepsilon) \|h\|_{C^0(\Omega)}.
  \end{equation}
Finally, letting $n \to +\infty$ in \eqref{eq:norm_difference_partial_sum}, the convergence of the series $\sum_{n=1}^{+\infty} g_n$ to $g$ in $C^1_c(\Omega)$ implies
\begin{equation}
    h(x) = Dg(x)
    \qquad
    \forall x \in K.
\end{equation}
\end{description}
\end{proof}

\begin{proof}[Proof of theorem \ref{t:main}]
    By Lemma \ref{lmm:quantitative_Lusin}, there exists $h \in C_c(\Omega)$ such that
        \begin{gather}
        \label{eq:mu_h_neq_f}
        \mu \left( \{x \in \Omega \colon h(x) \neq f(x)\} \right) < \frac{\varepsilon}{2},
        \\[5pt]
        \| h\|_{L^p(\Omega)} \leq \left(1+ \frac{\varepsilon}{2} \right) \|f\|_{L^p(\Omega)}
        \qquad
        \forall p \in [1,+\infty].
    \end{gather}
    If $h \equiv 0$, then $g \equiv 0$ is the desired function; otherwise we fix $N \in \N$ such that
    \begin{equation}\label{eq:choice_eta_last}
         \eta
         \coloneqq
         \frac{\|h\|_{C^0(\Omega)}}{N}
         <
         \frac{\varepsilon}{4} \inf_{p \in [1,+\infty]} \left\{ \frac{\|h\|_{L^p(\Omega)}}{\mu(\Omega)^{\frac{1}{p}}} \right\},
    \end{equation}
    and such $N$ exists since the above infimum is strictly positive (by continuity of the map $p \mapsto \|h\|_{L^p(\Omega)}/\mu(\Omega)^{\frac{1}{p}}$).
    We define $N$ functions $h_1, \dots, h_N \in C_c(\Omega, \R^n)$ as follows:
    \begin{equation}
        h_i(x)=
        \begin{cases}
            0                                           & \text{if } |h(x)| \leq (i-1)\eta
            \\[5pt]
            h(x) - (i-1)\eta \frac{h(x)}{|h(x)|}        & \text{if } (i-1)\eta < |h(x)| \leq i \eta
            \\[5pt]
            \eta  \frac{h(x)}{|h(x)|}                   & \text{if } |h(x)|> i \eta.
        \end{cases} 
    \end{equation}
    Hence the functions $h_i$ have the following properties:
    \begin{gather}
        \supp h_{N} \subset \subset \supp h_{N-1} \cdots \subset \subset  \supp h_1=\supp h,
        \\[4pt]
        0 \leq h_i \leq \eta \qquad \forall i=1, \dots, N,
        \\[4pt]
        h = h_1+\dots +h_N,
    \end{gather}
    Now for every $i=1, \dots, N$ we apply Lemma \ref{lmm:approx_continuous} to $h_i$, obtaining a function $g_i \in C^1_c(\Omega)$ and a compact set $K_i \subset \supp h_i$ which satisfies the conclusions of the lemma with $\frac{\varepsilon}{2N}$ in place of $\varepsilon$. Let $C \coloneqq \bigcap_{i=1}^N K_i$.

    We define $g= \sum_{i=1}^N g_i$ and we claim that $g$ satisfies the conclusions of the theorem.

        \item First of all we have $f=Dg$ on $C \cap \{h=f\}$ and
        \begin{equation}
            \mu\big( \Omega \setminus (C \cap \{h=f\}) \big)
            \leq
            \mu( \Omega \setminus C) + \mu\big( \{h \neq f\}\big)
            <
            \varepsilon.
        \end{equation}
        By interior approximation, we can find a compact set $K \subset (C \cap \{h=f\})$ such that $\mu(\Omega \setminus K) < \varepsilon$.

        It remains to estimate the $L^p$-norms of $Dg$: to this aim, we first claim that
        \begin{equation}\label{eq:estimate_below_h_layers}
            |Dg(x)| \leq |h(x)| + 2\eta
            \qquad
            \forall x \in \Omega.
        \end{equation}
        Indeed, first of all clearly $|Dg(x)| \leq |h(x)| + 2\eta$ for every $x \in \Omega \setminus \supp h_1$ since $\supp g \subseteq \supp h= \supp h_1$ by construction.
        Moreover, for every $x \in \supp h_{i} \setminus \supp h_{i+1}$ it holds
        \begin{equation}\label{eq:h_large}
            |h(x)| \geq (i-1)\eta,
          \end{equation}
        while $g_j(x)=0$ for every $j>i$ (since $\supp g_j \subset \supp h_j$) implies
        \begin{equation}\label{eq:g_small}
            |Dg(x)| \overset{(\supp g_j \subset \supp h_j)}{\leq} \sum_{j=1}^i |D g_j(x)|
            \leq
            i \left( 1+ \frac{\varepsilon}{2N} \right) \eta
            \overset{(\varepsilon<1)}{\leq}
            (i+1)\eta.
        \end{equation}
        \eqref{eq:h_large} and \eqref{eq:g_small} prove \eqref{eq:estimate_below_h_layers}, that in turn implies
        \begin{equation}
            \|Dg\|_{L^p(\Omega)}
            \leq
            \|h\|_{L^p(\Omega)} + 2\eta \big( \mu(\Omega) \big)^{\frac{1}{p}}
            \overset{\eqref{eq:choice_eta_last}}{<}
            \left( 1+ \frac{\varepsilon}{2} \right) \|h\|_{L^p(\Omega)}
            \qquad
            \forall p \in [1,+\infty),
        \end{equation}
        concluding the proof.
\end{proof}
\section{flat chain conjecture}\label{s:metric}
The question on which differential operator admit a Lusin type theorem in the spirit of\cite{alberti} has recently drawn some interest, see \cite{DeMaGa}. For general data, it is natural to expect that the Lipschitz constant of the Lusin-type approximated solutions blow up as $\varepsilon\to 0$. But, as in Theorem \ref{t:main2}, one could hope to keep the Lipschitz constant bounded if the given datum is trivial when restricted to the linear subspace along which the operator is closable.

We conjecture that the following version of Theorem \ref{t:main2} holds, where $f$ is replaced by a $k$-form. Here $V^k$ is the equivalent of the decomposability bundle for $k$-vectors, see \cite{AlbMar2}. We denote by $I(h,n)$ the set of multi-indexes of length $h$ in $\R^n$, i.e. the set of arrays $(i_1,\dots,i_h)$ with $1\leq i_1<i_2,\dots<i_h\leq n$.

\begin{conjecture}\label{conj}
    Let $\mu$ be a Radon measure on $\R^n$, let $\Omega \subset \R^n$ be an open set such that $\mu(\Omega)< \infty$ and let $\omega \colon \Omega \to \Lambda^k(\R^n)$ be a Borel $k$-form satisfying $$\langle\omega(x),\tau\rangle=0\qquad\mbox{ for every $\tau\in V^k(\mu,x),$ for $\mu$-a.e.\ $x \in \Omega$.}$$ 
    For every $\varepsilon>0$ there exists a compact set $K \subset \Omega$ and a $(k-1)$-form $\phi=\sum_{I\in I(k-1,n)}\phi_Idx_I$ in $\Omega$ of class $C^1$ such that
    \begin{gather}
        \mu(\Omega \setminus K) < \varepsilon,
        \\[4pt]
d\phi(x)=\omega(x) \qquad \forall x \in K,
        \\[4pt]
         \lip(\phi_I)\leq C(n)\|\omega\|_\infty, \quad\mbox{for every $I\in I(k-1,n)$}.
    \end{gather}
\end{conjecture}
We note that a positive solution of the above conjecture would immediately imply the validity of the Ambrosio-Kirchheim flat chain conjecture about the equivalence of metric currents in $\R^n$ and Federer-Fleming flat chains with finite mass, see \cite{AK}. In particular, since Theorem \ref{t:main} implies Conjecture \ref{conj} in the case $k=1$, it also implies the (already known) 1-dimensional version of the flat chain conjecture, see also \cite{Schioppa, alberti-bate-marchese, MarMer}. The conjecture was also proved in dimension $k=n$ in \cite{DPR}. Here we show the implication between the two conjectures, referring to \cite{MarMer} for the relevant notation. We only remind the reader that $\tilde{T}$ represents the classical current associated with a metric current $T$ (of finite mass).

\begin{proof}[Proof of the flat chain conjecture assuming the validity of Conjecture \ref{conj}]

Towards a proof by contradiction of the flat chain conjecture, assume that there exists a non trivial metric $k$-current $T=\tau\mu$ in $\R^n$ such that $\tilde T$ is not a flat chain. By subtracting the flat chain $(P_{V^k}\tau)\mu$, we may assume that $\tau(x)\in V^k(\mu,x)^\perp$ for $\mu$-a.e. $x$. Here $P_{V^k}:\Lambda_k(\R^n)\to\Lambda_k(\R^n)$ is the orthogonal projection on $V^k$.  

As shown in the proof of Theorem 4.4 of \cite{MarMer}, $\mu$ is singular with respect to Lebesgue, and by Proposition 4.2 of \cite{MarMer}, there exists a sequence of vectors $v_j\in\R^n$ such that $v_j\to 0$ and, denoting, for every $w\in\R^n$, $\tau_w:\R^n\to\R^n$ the map $x\mapsto x+w$, the currents $\tilde T$ and $(\tau_{v_j})_\sharp \tilde T$, are mutually singular for every $j\in\N$. In particular, $$\mathbb{M}(\tilde T-(\tau_{v_j})_\sharp \tilde T)=2\mathbb{M}(\tilde T)\qquad\mbox{ for every $j\in\N$}.$$

For every $j\in\N$, let $\omega_j$ be a Borel $k$ form such that $\|\omega_j\|_\infty\leq 1$ and $$\langle \tilde T-(\tau_{v_j})_\sharp \tilde T, \omega_j\rangle=2 \mathbb{M}(\tilde T).$$
In particular for every $j\in\N$, $$\langle\omega_j,\tau\rangle=0\qquad\mbox{ for every $\tau\in V^k(\mu,x),$ for $\mu$-a.e.\ $x$}.$$

Fix $\varepsilon>0$ sufficiently small and let $\phi_j$ be the corresponding $(k-1)$-forms obtained applying Conjecture \ref{conj} with $\mu=|\tilde T-(\tau_{v_j})_\sharp \tilde T|$. Possibly subtracting a constant, we can assume that $\phi_j(0)=0$, for every $j$, so that we can find a $C(n)$-Lipschtz $(k-1)$-form $\phi_\infty$ such that, up to non-relabeled subsequences,
    \begin{equation}
 \phi_{j}\to \phi_{\infty}\quad\text{locally uniformly.}
        \label{convdebole*}
    \end{equation}
    
   We deduce that for every $j \in \N$
   \begin{equation*}
          \begin{split}
2\mathbb{M}(\tilde T)=& \langle\tilde T-{(\tau_{v_j})}_\sharp\tilde T,\omega_j\rangle=\langle\tilde T,\omega_j\rangle-\langle{(\tau_{v_j})}_\sharp\tilde T,\omega_j\rangle\\
     =&\langle\tilde T,\omega_j\rangle-\langle{\tilde T,(\tau_{v_j})}^\sharp\omega_j\rangle
     \leq
     \langle\tilde T,d\phi_j\rangle-\langle{\tilde T,(\tau_{v_j})}^\sharp d\phi_j\rangle-C(n)\binom{n}{k-1}\varepsilon\\ =&T(1,\phi_j)-T(1,\phi_j\circ\tau_{v_j})-C(n)\binom{n}{k-1}\varepsilon.
       \end{split}    
   \end{equation*}
Thanks to the continuity of metric currents, the equi-Lipschitzianity of the $\phi_j$'s, and \eqref{convdebole*}, the quantity $T(1,\phi_j)-T(1,\phi_j\circ\tau_{v_j})$ tends to $0$  as ${j \to +\infty}$ and we reach a contradiction, if $\varepsilon$ is chosen sufficiently small.
\end{proof}

\begin{remark}
\begin{itemize}
    \item Thanks to Proposition 3.1 and Proposition 3.2 of \cite{MarMer}, in order to prove the flat chain conjecture in full generality, it would suffice to prove Conjecture \ref{conj} for closed forms $\omega$.
    \item In the previous proof we did not need the full validity of the equality 
    $$d\phi(x)=\omega(x)\qquad \forall x\in K,$$
    but only that, for a given $k$-vector field $\tau(x)\in V^k(\mu,x)^\perp$, the $(k-1)$-form $\phi$ yields
    $$\langle d\phi(x),\tau(x)\rangle\geq \langle\omega(x),\tau(x)\rangle\qquad\forall x\in K.$$
\end{itemize}
\end{remark}

\section*{Conflict of interest and data availability}
On behalf of all authors, the corresponding author states that there is no conflict of interest.\\

The manuscript has no associated data.

\addcontentsline{toc}{section}{\refname}
\printbibliography

\end{document}